\documentclass[12pt]{amsart}
\usepackage{amssymb,amsfonts,amsthm,amsmath,enumerate}
\usepackage{mathrsfs,dsfont}
\usepackage[english]{babel}
\usepackage[T1]{fontenc}
\usepackage{times, mathptm}

\newtheorem{lemma}{Lemma}[section]
\newtheorem{corollary}[lemma]{Corollary}
\newtheorem{definition}[lemma]{Definition}
\newtheorem{proposition}[lemma]{Proposition}
\newtheorem{remark}[lemma]{Remark}
\newtheorem{theorem}[lemma]{Theorem}

\newcommand{\field}[1]{\mathbb{#1}}
\newcommand{\N}{\field{N}}
\newcommand{\R}{\field{R}}
\newcommand{\Sr}{\field{S}}
\newcommand{\Rn}{\R^n}

\newcommand{\B}{\field{B}}
\newcommand{\Cal}[1]{\mathcal{#1}}

\newcommand{\Scr}[1]{\mathscr{#1}}
\newcommand{\nf}{\nabla f}
\newcommand{\nuf}{\nu_f}
\newcommand{\ud}{{\rm{d}}}
\newcommand{\dS}{\ud_\Sr}
\newcommand{\vol}{{\rm{vol}}}
\newcommand{\clos}{{\rm{\bf clos}}}
\newcommand{\crit}{{\rm{\bf crit}}}
\newcommand{\hes}{{\rm Hess}}

\newcommand{\jac}{{\rm{\bf Jac}}}
\newcommand{\fl}{\rightarrow}

\newcommand{\bs}{ {\tiny $\blacksquare$} \\}
\numberwithin{equation}{section}

\newenvironment{myproof}{\noindent{\it Proof.}
\setlength{\parindent}{0mm}} {$\hfill$ \bs}

\begin{document}
\title[On the total curvatures of a tame function]{On the total curvatures of a tame function}

%    Information for author
\author[V. Grandjean]{Vincent Grandjean}

\address{{\it Permanent Address:} V. Grandjean, Department of Computer Science,
University of Bath, BATH BA2 7AY, England,(United Kingdom)}

\address{{\it Current Address:} V. Grandjean, Fakul\"at V, Institut f\"ur Mathematik Carl
von Ossietzky Universit\"at, Oldenburg, 26111 Oldenburg i.O. (Germany)}

\email{cssvg@bath.ac.uk}

\thanks{Partially supported by the European research network IHP-RAAG
contract number HPRN-CT-2001-00271}

%    General info

\begin{abstract} 
Given a definable function $f:\Rn \mapsto \R$, enough differentiable, we study the continuity of the total
curvature function $t \to K(t)$, total curvature of the level $f^{-1} (t)$, and the total absolute
curvature function $t \to |K| (t)$,  total absolute curvature of the level $f^{-1} (t)$. We show they
admits at most finitely many discontinuities.  
\end{abstract}
\maketitle

\section{Introduction}
One of the nicest feature of o-minimal structures expanding the ordered field of real numbers is
that taking the derivative of a definable function (in a given such o-minimal structure) provides
a definable function (in the same given o-minimal structure). Unfortunately the inverse operation
is somehow much more delicate, integration and measure lead to problems even in some of their simplest
aspects.

In the setting of subanalytic geometry some measure theoretical aspects, density, Lipschitz-killing
curvatures, of the subanalytic sets had been studied \cite{BB, Com, Fu}. Still in this context,
the $k$-dimensional volume of a global subanalytic subset of $\Rn$, lying in a globally subanalytic
family of subsets of $\Rn$ of dimension at most $k$, is, when finite, a log-analytic function in the
parameter, as proved in \cite{LR,CLR}. This is already an issue since the logarithmic contribution
cannot be avoided and so the functions carrying the quantitative aspect of the variation of the volume
in the parameter of the family are already outside of the structure.

In the world of non-polynomially bounded o-minimal structures expanding the real
numbers, almost nothing similar to the statement in the globally subanalytic context is known in whole
generality. Let us nevertheless mention the results of  \cite{Le,Ka}, proved independently, namely the
definability of the set of the parameters at which the $2$-dimensional volume of a definable family
of plane definable subsets is finite.

\medskip
Let us go closer to the goal of this note. Given a definable function $f : \Rn \mapsto \R$, that is $C^l$
with $l \geqslant 2$, to each regular level $t$ we associate two real numbers, namely, $K(t)$, the total
curvature of the level $f^{-1} (t)$ oriented by the unitary gradient field of $f$, and, $|K|(t)$, the total
absolute curvature of the level $f^{-1} (t)$. When $n$ is odd and $f^{-1} (t)$ is compact and connected,
the Gauss-Bonnet-Chern Theorem states that $K(t)$ is just the Euler Characteristic of $f^{-1} (t)$ (modulo
a constant depending only on $n$). Somehow $K(t)$ and $|K| (t)$ have a connection with the topological
equisingularity type of $f^{-1} (t)$, even when  $f^{-1} (t)$ is no longer compact. Thus knowing the variation
as a function of $t$ of these total curvatures functions could give some information about the
equisingularity of the family of the levels. This was the first motivation of this study 
(see \cite{Gr} for some results in this direction).

From the measure theoretical point of view, these total curvatures are just weighted $n-1$-dimensional
volumes of a definable $1$-parameter family of subsets of $\Sr^{n-1}$. So how do the total curvature
functions in the parameter behave ? As already said, for the level of generality we want to deal with,
there is no hope (yet!) to provide some quantitative information about the variation of these functions
of $t$. Nevertheless we propose in this note to study some qualitative properties of the functions
$t \to K(t)$ and $t \to |K| (t)$, and we will find some.
We actually proved

\medskip
\noindent
{\bf Theorem.} {\em  Let $f : \Rn \mapsto \R$ be a $C^l$, $l\geqslant 2$, definable function. \\
(1) the function $t \to |K|(t)$ admits at most finitely many discontinuities. \\
(2) If the function $t \to |K|(t)$ is continuous at a regular value $c$ of $f$, so is
the function $t \to K(t)$.
}

\medskip
The paper is organized as follows.

In Section \ref{sectionNC} we provide some definitions, conventions and notations that will be used in the rest
of the note.

Section \ref{sectionCACH} recalls what the total curvature and the total absolute curvature
of a hypersurface are and what are the connections with linear orthogonal projections onto
oriented lines.

In Section \ref{sectionGMTF} we just define the Gauss map of a given function and states some of its
elementary properties in the frame established in Section \ref{sectionCACH}.

Section \ref{sectionHLGI} is devoted to Hausdorff limits of Gauss images since they will be the key
tool of our main result.

 Theorem \ref{thmCCF1} and Proposition \ref{propCCF4} are the main results of Section \ref{sectionCCF},
and  are proved with the help of some preliminary work.

 In Section \ref{sectionTLCTALC} we state results of the same flavor as those of Section \ref{sectionCCF},
but for what we named the total $\lambda$-curvature and total absolute $\lambda$-curvature.

We finish this paper with some remarks in Section \ref{sectionRCS}.
%
%
%
%
%
%
%
%      **********************************************************
%
%
%
%
%
%
%
\section{Notation - convention}\label{sectionNC}

Let $\Rn$ be the real $n$-dimensional affine space endowed with its Euclidean metric.  The scalar product
will be denoted by $\langle \cdot , \cdot \rangle$.

Let $\B_R^n$ be the open ball of $\R^n$ centered at the origin and of radius $R >0$.

Let $\Sr_R^{n-1}$ be the $(n-1)$-sphere centered at the origin and of radius $R>0$.

Let $\Sr^{n-1}$ be unit ball of $\R^n$ endowed  with the induced Euclidean metric and let $\dS$ be the
intrinsic distance function on $\Sr^{n-1}$.

Let $\ud v_k$ be the $k$-dimensional Hausdorff measure of $\R^n$ with $k \in \{1,\ldots,n\}$.

\medskip
Let us recall briefly what an o-minimal structure is.

\medskip
An {\it o-minimal structure $\Cal{M}$ expanding the ordered field of real numbers} is a collection
$(\Cal{M}_p)_{p \in \N}$, where $\Cal{M}_p$ is a set of subsets of $\R^p$ satisfying the following axioms

\smallskip
\noindent
1) For each $p\in \N$, $\Cal{M}_p$ is a boolean sub-algebra of subsets of $\R^p$. \\
2) If $A \in \Cal{M}_p$ and $B \in \Cal{M}_q$, then $A \times B \in \Cal{M}_{p+q}$. \\
3) If $\pi: \R^{p+1} \mapsto \R^p$, is the projection on the first $p$ factors, given any $A \in
\Cal{M}_{p+1}$, $\pi (A) \in \Cal{M}_p$.  \\
4) The algebraic subsets of $\R^p$ belongs to $\Cal{M}_p$. \\
5) $\Cal{M}_1$ consists exactly of the finite unions of points and intervals.

So the smallest o-minimal structure is the structure of the semi-algebraic subsets.

\medskip
Assume that such an o-minimal structure $\Cal{M}$ is given for the rest of this article.

\medskip
A subset $A$ of $\R^p$ is a {\it definable subset} (in the given o-minimal structure) of $\R^p$, if $A \in
\Cal{M}_p$.

A mapping $g: X \mapsto \R^q$, where $X \subset \R^p$, is a {\it definable mapping} (or just definable, for
short) if its graph is a definable subset of $\R^{p+q}$.

The reader may refer to \cite{Cos,vdD1,vdDM}  to learn more on the properties of definable subsets and
definable mappings.

\medskip
 let $Z$ be a connected definable subset of $\R^n$.
The dimension of $Z$, $\dim Z$, is well defined.

A point  $z_0 \in Z$ is smooth if there exists a neighborhood $U \subset \R^n$ of $z_0$, such that $U \cap
Z$ is diffeomorphic to $\R^k$, for an integer $k \leqslant \dim Z$. The property of being smooth of a given
dimension is a locally open property once $Z$ is equipped with the induced topology.

A point $z_0$ which is not smooth is called singular.
The set of such points is definable.

\medskip
Let $S_1$ and $S_2$ be respectively $C^1$ definable submanifolds of $\R^n$ and of $\R^p$. Let $g : S_1
\mapsto S_2$ be a $C^1$ definable function. A point $x_0$ is a smooth or regular point of $g$ if the rank
$d_{x_0}g$ is maximum in a neighborhood of $x_0$.

A critical or singular point $x_0$ of $g$ is a point at which
$d_{x_0}g$ is not of maximum rank. The set of such points is
definable, denoted by $\crit (g)$.

\medskip
By abuse of language, we will talk about the rank of the mapping $g$
at a point $x_0$ to mean the rank of the differential $\ud_{x_0}g$.

\medskip
A {\it definable family} $\Cal{E} = (E_t)_{t\in T}$ of subsets of a definable submanifold $S \subset \Rn$,
with parameter space $T\subset \R^m$ does not only mean that $E_t$ is a definable subset of $S$, but that the
subset $\{(x,t) \in S \times T : x \in E_t \}$ is definable. Equivalently it means it is the family
of the (projection onto $S$ of the) fibers  of a definable mapping.
%
%
%
%
%
%
%
%      **********************************************************
%
%
%
%
%
%
%
\section{Total curvature and total absolute curvature of a connected orientable hypersurface}
\label{sectionCACH}

\medskip
Let $M$ be a definable connected, $C^2$ hypersurface of $\R^n$. Assume $M$ is orientable and the
orientation is given by a  $C^1$ map, $M \ni x \to \nu_M (x) \in \Sr^{n-1}$. The map $\nu_M$ is
definable.

\medskip
Let $\Cal{U} \subset \Sr^{n-1}$  be defined as $\nu_M(M \setminus \crit(\nu_M))$, where
$\crit(\nu_M)$ is the set of critical points of $\nu_M$. The subset $\Cal{U}$ is a definable open
subset and, when not empty,
for any $u\in \Cal{U}$, $\nu_M^{-1}(u)$ is finite. By Gabrielov uniformity theorem, there exists a positive
integer $N_M$, such that $\#\nu_M^{-1}(u) \leqslant N_M$ for any $u \in \Cal{U}$.

\medskip
Since $\Cal{U}$ is a finite disjoint union of open definable subsets $U_i$, $i = 1,\ldots,d$, let $M_i =
\nu_M^{-1}(U_i) \setminus \crit(\nu_M)$ which is a definable subset of $M$. Let $s(i)$ be the number of
points in a fiber above any $u \in U_i$.

\medskip
Let $k_M(x)$ be the Gauss curvature at $x \in M$, that is the value of the determinant of the Jacobian
matrix of $\nu_M$. We make the convention that the $(n-1)$-dimensional volume of the empty set is $0$.

\begin{proposition}\label{propCACH1}
The total absolute curvature of $M$ is
\begin{center}
\smallskip
$|K|_M := \displaystyle{\int_M |k_M (x)| \ud v_{n-1} (x) = \sum_i s(i) \vol_{n-1}(U_i)}$.
\end{center}
\end{proposition}
\begin{myproof}
If the maximal rank of $\nu_M$ is at most $n-2$, then $\crit (\nu_M) =M$ and so we deduce $|K|_M =0$.

So assume that $\nu_M$ is at most $n-1$.

Let us denote by $\jac$, the Jacobian of any mapping (when it makes sense). Since the mapping $\nu_M$ is of
maximal rank, the set of its critical values of $\nu_M$, that is $\nu_M (\crit(\nu_M))$, is of codimension
at least $1$ in $\Sr^{n-1}$. Each $U_i$ is open and so is each $M_i$ in $M$.

We assume first that each connected component of $M_i$ is simply connected. \\
Then $M_i = M_i^{(1)} \sqcup\ldots \sqcup M_i^{s(i)}$ and $\nu_M$ induces a diffeomorphism from $M_i^{(j)}$
onto $U_i$, for each $j=1,\ldots,s(i)$. Thus we find
\begin{center}
\smallskip
$\displaystyle{\vol_{n-1}(U_i) = \int_{U_i} \ud v_{n-1} =\int_{M_i^{(j)}} \jac(\nu_M (x))\ud v_{n-1} (x)}$,
for each $j$.
\smallskip
\end{center}
Thus we deduce
\begin{center}
\smallskip
$\displaystyle{\int_{M_i}} \jac(\nu_M (x))\ud v_{n-1} (x) = s(i)\vol_{n-1}(U_i)$.
\smallskip
\end{center}
As another consequence of this fact, if we define $M^*$ as $M \setminus (\sqcup_i M_i )$, then $M^*$ is
definable, and since $\nu_M (M^*)$ is at most of dimension $n-2$ (meaning that $\jac(\nu_M)$ is zero on a
definable dense open set of $M^*$ if dim $M^* = n-1$), we obtain
\begin{center}
\smallskip
 $\displaystyle{\int_{M^*}} \jac(\nu_M (x))\ud v_{n-1} (x)=0$.
 \smallskip
\end{center}
Then
\begin{center}
\smallskip
$|K|_M= \displaystyle{\int_M \jac(\nu_M (x))\ud v_{n-1} (x)= \sum_i \int_{M_i} \jac(\nu_M (x))\ud v_{n-1}
(x)}$,
\smallskip
\end{center}
which is the desired result.

\medskip
In the general situation, by the cylindrical decomposition theorem, there exists a closed subset $N_i$ of
$M_i$ of dimension at most $n-2$, such that each connected component of $M_i \setminus N_i$ is simply
connected. Then we do the same as above with $\nu_M (M_i \setminus N_i)$ instead of $U_i$. Since
$\vol_{n-1}(\nu_M (N_i)) = 0$, the formula given is still true in this general case
\end{myproof}

For each $i=1,\ldots,m$, let $\sigma_i^+$ be $\# (\nu_M^{-1}(u) \cap \{\det k_M >0\})$ and $\sigma_i^-$ 
be $\#(\nu_M^{-1}(u) \cap \{\det k_M <0\})$, for any $u \in U_i$, 
since these numbers depend only on $U_i$. We deduce that $s(i) =  \sigma_i^+ + \sigma_i^-$ 
and thus

\begin{center}
\smallskip
$|K|_M = \sum_i^{N_M} (\sigma_i^+ + \sigma_i^-)\vol_{n-1}(U_i)$.
\smallskip
\end{center}

\medskip
\noindent
 Let $\sigma_i = \sigma_i^+ - \sigma_i^-$. Note that $\sigma_i = \deg_u \nu_M = :
\deg(\nu_M|_{\nu_M^{-1}(U_i)})$ the degree of the mapping $\nu_M$ at any $u \in U_i$.

\begin{proposition}\label{propCACH2}
The total curvature of $M$ is
\begin{center}
\smallskip
$K_M := \displaystyle{\int_M k_M (x) \ud v_{n-1} (x)} = \sum_{i=1}^{N_M} \sigma_i\vol_{n-1}(U_i)$.
\smallskip
\end{center}
\end{proposition}
\begin{myproof}
It works exactly as in the proof of Proposition \ref{propCACH1}.
\end{myproof}

Let us come to a more specific property of the Gauss map. For any $u\in \Sr^{n-1}$, let $\varphi_u (x) :=
\langle x,u \rangle$ be the orthogonal projection on the oriented vector line $\R u$.

Let us consider the following
\begin{lemma}\label{lemCACH1}
Let $y \in M$ be a point at which the rank of $\ud_y \nu_M$ is $n-1$. Let $u = \nu_M (y)$. Then the
function $\varphi_u|_M$ has a Morse singular point at $y$, that is the Hessian matrix $\hes_y
(\varphi_u|_M)$ is non degenerate.
\end{lemma}
\begin{proof}
Let us identify the hyperplane $T_yM$ with $\R^{n-1}$ with coordinates ${\bf x} = (x_1,\ldots,x_{n-1})$
centered at $y$ and remember that $\ud_y \nu_M$ can be seen as a reflexive endomorphism of $T_yM =
T_u\Sr^{n-1}$. Since it is of rank $n-1$, it has exactly $´(n-1)$ non-zero real eigenvalues.

By the definable implicit function theorem, there exist an open neighborhood $\Scr{U}$ of the origin in
the hyperplane $T_yM$ and a $C^1$-definable map $\phi : \Scr{U} \mapsto \R$ such that there exists a
neighborhood $\Scr{V}$ of $y$ in $M$ such that $\Scr{V} = \{x_n = \varphi({\bf x}) : {\bf x} \in
\Scr{U}\}$. Thus we find that $\varphi_u|_M (x) = \phi({\bf x})$. We thus rewrite $\nu_M$ as
\begin{center}
\smallskip
$\displaystyle{\nu_M (x) = \nu_M ({\bf x}) = \frac{\nabla(x_n - \phi ({\bf x}))}{|\nabla(x_n - \phi ({\bf
x}))|}}$.
\smallskip
\end{center}
For any $u \in T_yM$ we find that
\begin{eqnarray*}
\ud_y \overline{\nu}_M \cdot \xi & = & \frac{1}{|\nabla(x_n - \phi ({\bf x}))|} \left [ \hes_y \phi \cdot
\xi - \langle \hes_y \phi \cdot \xi,u\rangle \; u\right]
\\
& = & \frac{1}{|\nabla(x_n - \phi ({\bf x}))|}
 \hes_y \phi \cdot \xi,
\end{eqnarray*}
where $\hes_y \phi$ is the Hessian matrix of $\phi$ at $y$. Since $\ud_y \nu_M$ is of rank $n-1$, we deduce that
$\hes_y (\varphi_u|_M) = \hes_y \phi$, and thus is Morse at $y$.
\end{proof}

Thanks to Lemma \ref{lemCACH1} we can then define the following
\begin{definition}\label{defCACH1}
The Morse index of $M$ at $x \notin \crit (\nu_M)$ is the integer $\lambda_M (x)$ define as the Morse index
at $x$ of the function $\varphi_u|_M$, with $u = \nu_M (x)$.
\end{definition}
Obviously this definition depends on the choice of $\nu_M$.

\medskip
Let $U$ be a connected component of $\Cal{U} = \nu_M (M) \setminus \nu_M(\crit (\nu_M))$. Thus $\nu_M$
induces a finite $C^{l-1}$ covering
\smallskip
\begin{center}
$\nu_M : M_1 \sqcup \ldots \sqcup M_d \mapsto U$.
\end{center}

\begin{proposition}\label{propCACH4}
With the previous notation, the function $M_j \ni x \to \lambda_M (x)$ is definable and so is constant.
\end{proposition}
\begin{proof}
Each subset $M_j$ is a definable and open in $M$. Given an orthonormal basis of $\Rn$, let $P(x;T) \in \R
[T]$ be the characteristic polynomial of $\hes_x (\varphi_{\nu_M (x)}|_M)$. The coefficients of this
polynomial are definable functions of $x \in M \setminus \crit (\nu_M)$ and $C^{l-1}$. As functions of $x$,
the roots (counted with multiplicity) are continuous and definable and all non zero. Since $M_j$ is
connected, the number of negative roots is constant and so is the index $\lambda_M (x)$.
\end{proof}

There is a straightforward corollary of Lemma \ref{lemCACH1} and Proposition \ref{propCACH4}
\begin{corollary}\label{corCACH1}
Let $y \in M$ be a point at which the rank of $\ud_y \nu_M$ is $n-1$. Let $u = \nu_M (y)$. Then the
function Morse index of $\varphi_u|_M$ and the index of $\ud_y \nu_M$ are the same.
\end{corollary}

Thus we can also rewrite the total curvature as

\bigskip
\begin{center}
$K_M  = \displaystyle{\int_{\nu_M(M) \setminus \nu_M(\crit (\nu_M))}}$ $\left(\sum_{x \in \nu_M^{-1}(u)\setminus
\crit (\nu_M)} (-1)^{\lambda_M (x)} \right)\; \; \ud v(u)$,
\end{center}

\medskip
since the set of points $\nu_M^{-1} (\crit (\nu_M)) \cap (M \setminus \crit(\nu_M))$ is at most of dimension
$n-2$.
%
%
%
%
%
%
%      **********************************************************
%
%
%
%
%
%
\section{On the Gauss map of a tame function}\label{sectionGMTF}

Let $f : \R^n \mapsto \R$, be a definable function, enough
differentiable, say $C^l$, with $l\geqslant 2$.

Let us denote by $K_0 (f)$ the set of critical values of $f$.

For each $t$, let $F_t$ be the level $f^{-1} (t)$.

The Gauss map of the function $f$ is the following mapping

\medskip
\begin{tabular}{rcccc}
\hspace{3cm} $\nuf$ &  :& $\R^n \setminus \crit (f)$ & $\mapsto$&
$\Sr^{n-1}$ \\
& & $x$ & $\to$ & $\displaystyle{ \frac{\nabla f (x)}{|\nabla f
(x)|}}$
\end{tabular}

\medskip
It is a definable and $C^{l-1}$ mapping. We will denote by $\nu_t$ the restriction $\nuf|_{F_t}$,
providing an orientation to each (connected component of the) level $F_t$ which is compatible with the
transverse structure of the foliation (on $\R^n \setminus \crit (f)$) by the connected components of
the (regular) levels of $f$. Note that $\crit(\nuf) \cap F_t = \crit(\nu_t)$.

\medskip
The differential mapping of $\nuf$ at $x \notin \crit (f)$, is

\smallskip
\begin{center}
$\ud_x \nuf : \xi \mapsto \displaystyle{\frac{1}{|\nabla f (x)|} [\hes_x(f) \cdot \xi - \langle \hes_x(f)
\cdot \xi, \nuf (x) \rangle \nuf (x)]}$
\end{center}

\medskip
We recall that for any $x \in F_t$, the linear mapping  $\ud_x \nu_t = \ud_x \nuf|_{T_xF_t}$ seen as an
endomorphism of $T_{\nu_f(x)} \Sr^{n-1}$ is reflexive.

\begin{definition}\label{defGMTF1}
Let $x \notin \crit (f) \cup \crit(\nuf)$. The tangent Gauss index of $f$ at $x$, denoted
by $\lambda_f (x)$ is the index of the reflexive endomorphism $\ud_x \nuf|_{T_xf^{-1}(x)}$.
\end{definition}

\begin{proposition}\label{propGMTF1}
The function $x \to \lambda_f (x)$ is a locally constant definable mapping.
\end{proposition}
\begin{proof} Assume a fixed orthonormal basis of $\Rn$ is given. Let $P(x,T) \in \R [T]$ be the characteristic
polynomial of the endomorphism $\ud_x \nuf|_{T_xf^{-1}(x)}$. From the computation of the differential
$\ud_x \nuf$ we deduce that the coefficients of this polynomial are definable and $C^{l-2}$. Given $x$, let
$\alpha_1(x) \leqslant \ldots \leqslant \alpha_{n-1}(x)$ be the roots of this polynomial. For each
$i=1,\ldots,n-1$, the functions $x \to \alpha_i (x)$ is continuous and definable. For this reason the
number of negative roots (that is the tangent Gauss index) is constant on each connected components of $\Rn
\setminus (\crit (f) \cup \crit(\nuf))$, that is $x \to \lambda_f (x)$ is constant on $\Rn \setminus  
(\crit (f) \cup \crit(\nuf))$. \\
Since the mapping $x \to \lambda_f (x)$ is constant on each connected component of $\Rn \setminus (\crit
(f) \cup \crit(\nuf))$, then its graph is a definable subset of $\Rn \times \R$ since $\Rn \setminus (\crit
(f) \cup \crit(\nuf))$ is a definable subset of $\Rn$.
\end{proof}

This property will be very useful in Section \ref{sectionCCF}.

\medskip
When $\crit(\nuf) \neq \emptyset$, let us consider the following $C^{l-1}$ definable mapping:

\smallskip
\begin{center}
$\Psi_f : \R^n \setminus (\crit (f) \cup \crit(\nuf))\mapsto \Sr^{n-1} \times \R $ defined as $x \to
(\nuf(x),f(x))$.
\end{center}

\medskip
This mapping is definable and $C^{l-1}$. It is also a local diffeomorphism at any of the point of $\R^n
\setminus (\crit (f) \cup \crit(\nuf))$, thus its image $\widetilde{\Cal{U}}$ is open.

\begin{definition}\label{defGMTF2}
Let $(u,t) \in \widetilde{\Cal{U}}$. The tangent Gauss degree of $f$ at $(u,t)$, denoted $\deg_f (u,t)$ is the
degree of $\Psi_f$ at $(u,t)$.
\end{definition}

Then we get the following,

\begin{proposition}\label{propGMTF2}
The function $(u,t) \to \deg_f (u,t)$ is a locally constant definable mapping.
\end{proposition}
\begin{proof}
This comes from the fact that $\Psi_f$ is a local diffeomorphism at each point of its definition domain.
\end{proof}

More interestingly we also have

\begin{proposition}\label{propGMTF3}
For any $(u,t) \in \widetilde{\Cal{U}}$,
\begin{center}
$\deg_{(u,t)} \Psi_f = \deg_u \nu_t = \displaystyle{\sum_{x \in \nu_t^{-1}(u)} (-1)^{\lambda_f (x)}}$.
\end{center}
\end{proposition}
\begin{proof}
Let $y \in \Psi_f^{-1}(u,c)$. Let $\Scr{F} :=(v_1,\ldots,v_{n-1},v_n)$, with $v_n = \nu_f$, be a $C^1$ and
direct orthonormal frame in a neighborhood $\Cal{Y}$ of $y$. With such "coordinates", for any $x \in 
F_t \cap \Cal{Y}$, both $\ud_x \Psi_f$ and $\ud_x \nu_t$ are considered as endomorphisms. For $x \in F_t \cap \Cal{Y}$, writing the matrices of $\ud_x \Psi_f$ and $\ud_x \nu_t$ in this frame, gives
\begin{center}
$\det ({\rm Mat}_{\Scr{F}(x)}\ud_x \Psi_f) \cdot \det({\rm Mat}_{\Scr{F}(x)} \ud_x \nu_t)>0$,
\end{center}
since the frame respects the orientations of $\Rn$ and of $T_xF_t$ and where ${\rm
Mat}_{\Scr{F}(x)}$ denotes the matrix in the base $\Scr{F}(x) :=(v_1(x), \ldots,v_{n-1}(x),v_n(x))$.
Thus the signs of these determinants are the same and so is proved the lemma.
\end{proof}
%
%
%
%
%
%
%
%
%      **********************************************************
%
%
%
%
%
%
%
\section{Hausdorff limits of Gauss images}\label{sectionHLGI}

We use notations of Section \ref{sectionGMTF}.

\medskip
We recall that $\widetilde{\Cal{U}} = \Psi_f (\R^n \setminus (\crit (f) \cup \crit (\nuf)))$ is open and
definable in $\Sr^{n-1} \times \R$ and that for any $(u,t) \in \widetilde{\Cal{U}}$, $\Psi_f^{-1}(u,t)$ is
finite.

From Gabrielov Uniformity theorem,  there exists a positive integer $N_f$ such that for any $(u,t)
\in \widetilde{\Cal{U}}$, $\#\Psi_f^{-1}(u,t) \leqslant N_f$.

Let us define the following definable sets

\medskip
\begin{tabular}{lll}
$\widetilde{\Cal{U}}_k$ & = & $\{(u,t) \in \widetilde{\Cal{U}} :
\#\Psi_f^{-1}(u,t) = k \}$.  \\
$\Cal{U}_t$ & = & $\{u \in \Sr^{n-1} : (u,t) \in \widetilde{\Cal{U}}\}$ = $
\nu_t (F_t \setminus \crit (\nuf)) = \nu_t (F_t \setminus \crit (\nu_t))$.
\\
$\Cal{U}_{k,t}$ & = & $\{u \in \Sr^{n-1} : (u,t) \in \widetilde{\Cal{U}}_k \}$.
\end{tabular}

\medskip
The subsets $\Cal{U}_t$ and $\Cal{U}_{k,t}$ are open in $\Sr^{n-1}$.

Note that at any point of $\R^n \setminus (\crit (f)\cup \crit(\nuf))$, $\Psi_f$ is a local diffeomorphism.

Let us remark that both families $(\Cal{U}_t)_{t\in f(\Rn) \setminus K_0 (f)}$ and
$(\Cal{U}_{k,t})_{t\in f(\Rn) \setminus K_0 (f)}$ are definable families of open subsets of
$\Sr^{n-1}$.

\bigskip
Let $\Cal{K}(\R^n)$ be the space of compact subsets of $\R^n$.

Given $Y$ and $Z$ compacts subsets of $\Rn$, the Hausdorff distance between $Y$ and $Z$, denoted by
$\ud_{\Cal{K}(\R^n)}(Y,Z)$, is defined as

\medskip
\begin{tabular}{rcl}
\vspace{6pt}
$\ud_{\Cal{K}(\R^n)}(Y,Z)$ & = & $\max(\min_{y\in Y}{\rm dist}(y,Z),\min_{z\in Z}{\rm dist}(z,Y))$
\\
\vspace{6pt}
& = & $\min \{r\geqslant 0 : \forall y \in Y \mbox { and } \forall z \in Z, \;  {\rm dist}(y,Z) \leqslant r $ \\
& & $ \hfill \mbox{ and } {\rm dist}(z,Y) \leqslant r\}$
\end{tabular}

\medskip
The space $\Cal{K}(\R^n)$ equipped with the Hausdorff distance $\ud_{\Cal{K}(\R^n)}$ becomes a 
complete metric space.

\medskip
Bounded definable families of compacts subsets behave well under the Hausdorff limit, (see
\cite{Br,vdD2,LS} for a general frame). Since we are only interested in  $1$-parameter families of such
subsets, the statement of the next result is given in this context and in the form we will use it below
in the rest of this paper.

\begin{theorem}[\cite{Br,vdD2,LS}]
Let $(\Cal{C}_t)_{t\in [0,1[}$ be a definable family of closed subsets of $\Rn$. Assume there exists a
compact subset $Q$ such that $\Cal{C}_t \subset Q$ for each $t\in[0,1[$. Then, the Hausdorff limit
$\Cal{C}_1 :=\lim_{t\fl 1}\Cal{C}_t$ does exist and is a definable subset of $\Rn$ contained in $Q$.
\end{theorem}

Let $(\Cal{C}_t)_{t \in I}$ be a definable family of closed subsets of $\Sr^{n-1}$. Since
there exists a constant $C$, such that given any $u,v \in \Sr^{n-1}$, $\ud_\Sr (u,v)\leqslant C |u-v|$, 
taking Hausdorff limits of closed subsets $(\Cal{C}_t)_{t \in I}$ of $\Sr^{n-1}$, will provide exactly
the same Hausdorff limits of $(\Cal{C}_t)_{t\in I}$ considered as closed subsets of $\R^n$.

\medskip
Let us return to our topic.
Given a value $c$ let us denote respectively by  $\Scr{V}_c^+$ and
by $\Scr{V}_c^-$ the following Hausdorff limits

\medskip
\begin{center}
$\Scr{V}_c^+ := \lim_{t\rightarrow c^+} \clos(\Cal{U}_t)$ and $\Scr{V}_c^- :=\lim_{t\rightarrow c^-}
\clos(\Cal{U}_t)$.
\end{center}

\medskip
For each $k = 1,\ldots,N_f$ and for $*=+,-$, let $\Scr{V}_k^*$ be
the Hausdorff limit $\lim_{t\rightarrow c^*} \clos(\Cal{U}_{k,t})$.

\begin{proposition}\label{propHLGI1}
Let $U_1,\ldots,U_{d_c}$ be the  connected components of $\Cal{U}_c$. For each $i=1,\ldots,d_c,$  there
exist positive integers $l_- = l_- (i) \geqslant s(i)$ and $l_+ = l_+ (i) \geqslant s(i)$ such that $U_i
\cap \Scr{V}_c^- \subset \Scr{V}_{l_-}^-$ and $U_i \cap \Scr{V}_c^+ \subset \Scr{V}_{l_+}^+$.
\end{proposition}
\begin{myproof}
Let $u \in \Sr^{n-1}$. When not empty, the subset

\medskip
\begin{center}
$\Psi_f^{-1}(\{u\}\times \R) =\{x \notin \crit (f) \cup \crit(\nuf) : \nuf (x) = u\}$
\end{center}
is a $C^{l-1}$ definable curve since $\Psi_f$ is a local diffeomorphism.
Let $\Gamma$ be a connected component of $\Psi_f^{-1}(\{u\})$ which intersects with $F_c$. Since
the function $f|_\Gamma$ is strictly monotonic, there exists $\varepsilon >0$ such that for any $t \in
]c-\varepsilon,c[\cup]c,c+\varepsilon[$, the curve $\Gamma$ also intersects with $F_t$, that is $u \in
\Cal{U}_t$. Since the number of connected components of $\Psi_f^{-1}(\{u\}\times \R)$ meeting $F_c$
does depend only on $U_i$, the Proposition is proved.
\end{myproof}

As a consequence of Proposition \ref{propHLGI1} we obtain

\begin{corollary}\label{corHLGI1}
Let $c$ be a value such that $\Cal{U}_c \neq \emptyset$. Then $\Cal{U}_c \subset \Scr{V}_c^+ \cap
\Scr{V}_c^-$.
\end{corollary}
\begin{myproof}
It is contained in the proof of Proposition \ref{propHLGI1}.
\end{myproof}

The statement of Corollary \ref{corHLGI1} would be wrong for a general definable family $(\Cal{C}_t)_t$ of
closed subsets of $\Sr^{n-1}$. As read in the proof, the transverse structure given by the function $f$ is
somehow also carried in the family $(\Cal{U}_t)_t$, and thus explain this result.

We end this section with an elementary result about the volume of a Hausdorff
limit of a $1$ parameter definable family.

\begin{proposition}\label{propHLGI2}
Let $\Cal{C} : = \cup_{t \in [0,1[}\Cal{C}_t$ be a definable family
of closed connected subsets of dimension $n$ of $Q \in \Rn$ or $\Sr^n$, a compact subset.
Let $\Cal{C}_1$ be the Hausdorff limit $\lim_{t\rightarrow 1} \Cal{C}_t$, then
\begin{center}
$\lim_{t\rightarrow 1} \vol_n (\Cal{C}_t) = \vol_n (\Cal{C}_1)$.
\end{center}
\end{proposition}
\begin{myproof}
According to \cite[Theorem 1]{LS} there exist at most finitely many subsets of $\Cal{K}(\R^n)$ that belong
to $\clos (\Cal{C}) \setminus \Cal{C} \subset \Cal{K}(\R^n)$. By \cite[Proposition 3,2]{vdD2}), $\Cal{C}_1$
is well defined and definable, and its dimension is at most $n$. Note that $\Cal{C}_1$ is necessarily
connected.

\medskip
\noindent
Assume first that $Q \subset \Rn$. \\
Let $T_\varepsilon (\Cal{C}_t)$ be the  $\varepsilon$-neighborhood of $\Cal{C}_t$
in $\Rn$ for $\varepsilon >0$.

Let $r$ be the dimension of $\Cal{C}_1$. Using a generalized Weyl-Steiner's tube formula for compact
definable subsets of $\Rn$ (\cite{BK}), we deduce that for small $\varepsilon >0$ there exists a positive
constant $L$ such that

\medskip
\begin{center}
$\displaystyle{\left |\vol_r (\Cal{C}_1) - \frac{\vol_n (T_\varepsilon\Cal{C}_1)}{\sigma_{n-r,n}
\varepsilon^{n-r}} \right |  < L \varepsilon}$,
\end{center}

where $\sigma_{n-r,n}$ is the volume of the unit ball $\B_{n-r}$ of $\R^{n-r}$.

\medskip
By definition of $\Cal{C}_1$, for any $\varepsilon >0$, there exists $\eta = \eta(\varepsilon)$
such that for $t>1-\eta$ we get $\ud_{\Cal{K}(\R^n)}(\Cal{C}_t,\Cal{C}_1) < \varepsilon$, that is
$\Cal{C}_t \subset T_\varepsilon (\Cal{C}_1)$.

\medskip
If $r=n$,  we deduce that $\vol_n (\Cal{C}_t) - \vol_n (\Cal{C}_1) \fl 0$ as $\varepsilon \fl 0$,
since $\Cal{C}_t \subset T_\varepsilon (\Cal{C}_1)$. \\
When $r<n$ we know $\vol_n (T_\varepsilon(\Cal{C}_1)) \fl 0$ as $\varepsilon \fl 0$. So $\vol_n
(\Cal{C}_t)$ must tend to $0$ as $t \fl 1$ (otherwise $\limsup_{t\fl 1} \vol_n (\Cal{C}_t) >0$, which would
contradict $\lim_{\varepsilon \fl 0}\vol_n (T_\varepsilon(\Cal{C}_1)) = 0$).

\medskip
Assume now $Q \subset \Sr^n$. The generalized Weyl's tube formula (\cite{BB},\cite{BK}) works also in
spaces of constant curvatures. Thus the proof above in the flat case, adapts almost readily to the case of
$\Sr^n$.
\end{myproof}
%
%
%
%
%
%
%
%      **********************************************************
%
%
%
%
%
%
%
\section{Continuity of the total curvature and of the total absolute curvature of a definable
function}\label{sectionCCF}

We use the notations of Section \ref{sectionGMTF}.

\medskip
Given $t\notin K_0 (f)$, let $k_t$ be the Gauss curvature of $F_t$ with respect to the orientation $\nu_t$.
Let $\Cal{E}_t$ be the set of the connected components of $F_t$. \\
When $t \notin K_0 (f)$, we define the {\it total absolute curvature}  of $F_t$ as

\smallskip
\begin{center}
$|K|(t) :=\sum_{E \in \Cal{E}_t} |K|_E $,
\end{center}

\medskip
\noindent
 and the {\it total curvature} of $F_t$ as

\smallskip
\begin{center}
$K(t) := \sum_{E \in \Cal{E}_t} K_E$.
\end{center}

\medskip
Assume $c$ is a value taken by $f$ such that $\Cal{U}_c$ is not empty. Let $U_1,,\ldots,U_{d_c},$ be the
connected components of $\Cal{U}_c$. For each $i=1,\ldots,d_c$, let  $s(i)$ be the number of points in the
fiber $\nu_c^{-1} (u)$ above any $u \in U_i$, that is $U_i \subset \Cal{U}_{s(i),c}$, and let $\sigma_i (c)$ 
be the degree of $\nu_c$ on $U_i$ (see Section \ref{sectionCACH}). We recall that

\begin{equation}\label{eqnTAC}
|K|(c) = \sum_{l=1}^{N_f} l\cdot \vol_{n-1} (\clos(\Cal{U}_{l,c})) = \sum_{i=1}^{d_c} s(i) \vol_{n-1}
(\clos(U_i)).
\end{equation}
\begin{equation}\label{eqnTC}
K(c) =  \sum_{i=1}^{d_c} \sigma_i (c)\vol_{n-1} (\clos(U_i)) \hfill
\end{equation}

As a consequence of the result of the existence of definable Hausdorff Limits (\cite{vdD2},\cite{LS}) and
of the definition of the total curvature we deduce the following
\begin{corollary}\label{corCCF1}
The following limits, 
\begin{center}
\smallskip
$\displaystyle{\lim_{t \fl c^-}K(t), \; \lim_{t \fl c^+}K(t), \; \lim_{t \fl c^-}|K|(t)}$ and $
\displaystyle{\lim_{t \fl c^+}|K|(t)}$ 
\smallskip
\end{center}
do exist, once given a value $c$.
\end{corollary}
\begin{proof}
It is just a consequence of the formulae for the total curvatures given above and of Proposition
\ref{propHLGI2}.
\end{proof}

\medskip
Let us recall that given $* = -,+$, we denote by $\Scr{V}_c^*$, the
Hausdorff limit $\lim_{t\rightarrow c^*} \clos(\Cal{U}_t)$ and given
any integer $1 \leqslant k \leqslant N_f$, $\Scr{V}_k^*$ stands for
the Hausdorff limit $\lim_{t\rightarrow c^*} \clos(\Cal{U}_{k,t})$

\smallskip
We can write $\Scr{V}_c^- \cap \Scr{V}_c^+$ as the partition
$\cup_{0\leqslant k,l \leqslant N_f} \Scr{V}_{k,l}$, where
$\Scr{V}_{k,l} = \Scr{V}_k^- \cap \Scr{V}_l^+$, if $k,l$ are both
positive, $\Scr{V}_{k,0}$ stands for $\Scr{V}_k^-$ if the
intersection with $\Scr{V}_c^+$ is empty, $\Scr{V}_{0,l}$ stands for
$\Scr{V}_l^+$ if the intersection with $\Scr{V}_c^-$ is empty, and
$\Scr{V}_{0,0} = \emptyset$.

Our goal is still to try to understand the continuity of the
functions $|K|$ and $K$. Regarding the results of Sections
\ref{sectionCACH} and \ref{sectionHLGI}, the non zero contribution
to $|K|(t)$ only comes from the subset $\Cal{U}_t$.

\begin{proposition}\label{propCCF3}
Let $c$ be a value. Then
\begin{center}
$|K|(c) \leqslant \min\{\lim_{t\rightarrow c^-}|K|(t),\lim_{t\rightarrow c^+} |K|(t)\}$.
\end{center}
\end{proposition}
\begin{myproof}
If $\Cal{U}_c$ is empty this means that $|K|(c) = 0$, and so the
statement is trivially true.

Let $U_1,\ldots,U_{d_c}$ be the connected components of $\Cal{U}_c$. For each $i \in\{1,\ldots,d_c\}$ 
and each $* = +,-$, let $l_*(i)$ be the corresponding integer of Proposition \ref{propHLGI1}. Observe 
that for $*=+$ and $*=-$
\begin{center}
$|K|(c) = \sum_i s(i) \vol_{n-1} (U_i) \leqslant \sum_i l_*(i) \vol_{n-1} (U_i \cap \Scr{V}_c^*)$,
\end{center}
and the right hand side term is $\leqslant \sum_l l
\vol_{n-1} (\Scr{V}_l^*)$. By Proposition \ref{propHLGI2}, we
conclude the proof.
\end{myproof}

The continuity of $|K|$ at $c$ implies that each of these inequalities is an equality. Using notation of
Proposition \ref{propHLGI1} we actually obtain

\begin{corollary}\label{corCCF2}
If $|K|$ is continuous at $c$, then \\
(1) for each pair $k,l$ with $k\neq l$, $\vol_{n-1}(\Scr{V}_{k,l})
=0$, or equivalently $\Scr{V}_{k,l}$ is of dimension at most $n-2$.
\\
(2) for each $k$ and each $*=+,-$, $\vol_{n-1}(\Scr{V}_k^*) =
\vol_{n-1}(\Cal{U}_{k,c})$.
\end{corollary}
\begin{myproof}
Using Proposition \ref{propHLGI1}, equation \ref{eqnTAC} and
Proposition \ref{propCCF3}, and writing down the continuity at $c$
provides the different statements.
\end{myproof}

Now we can come to the first main result of this section
\begin{theorem}\label{thmCCF1}
The function $|K|$ has at most finitely many discontinuities.
\end{theorem}
\begin{myproof}
This is a consequence of Proposition \ref{propHLGI2} and Proposition \ref{propHLGI1}.

Assume $k$ is given. Let $\Cal{C} : = \cup_{t\in f(\R) \setminus K_0 (f)} \Cal{C}_t$, where
$\Cal{C}_t = \clos(\Cal{U}_{k,t})$.

Assume first that $f(\R) \setminus K_0 (f)$ is connected, so that we can assume that $f(\R)
\setminus K_0 (f) = ]0,1[$. Then $\Cal{C}$ is a definable family of subsets of $\Sr^{n-1}$, and then
is definable in $\Cal{K}(\R^n)$. Thus its closure $\clos(\Cal{C})$ in $\Cal{K}(\R^n)$ and consists of
$\Cal{C}$ and finitely many closed definable subsets of $\Sr^{n-1}$ by \cite[Theorem 1]{LS}.
This means there exist at most finitely many $c \in f(\R) \setminus K_0 (f) = ]0,1[$ such that the
Hausdorff limit $\lim_{t\rightarrow c^-}\Cal{C}_t$ or the Hausdorff limit $\lim_{t\rightarrow c^+}\Cal{C}_t$
is not $\Cal{C}_c$. For any value $c \in ]0,1[$ such that $\lim_{t\rightarrow c^-}\Cal{C}_t =
\lim_{t\rightarrow c^+}\Cal{C}_t = \Cal{C}_c$, Proposition \ref{propHLGI2} states that $t \to
\vol_{n-1}(\Cal{C}_t)$ is continuous at such a $c$. Since there are only finitely many $k$ and finitely
many connected components of $f(\R) \setminus K_0 (f)$, the theorem is proved.
\end{myproof}

\begin{remark}\label{rmkCCF1}
Working with globally subanalytic functions only, Theorem \ref{thmCCF1} is then just
a consequence of Lion-Rolin's Theorem \cite{LR}.
\end{remark}

Now let us investigate the continuity of $K$, $t \to K(t) = \int_{F_t} k_t$.

Let us pick a value $c$. Let $\{U_i\}_{i=1}^{d_c}$ be the set of connected components of $\Cal{U}_c$. We
recall that $\deg_u \nu_c$, the degree of $\nu_c$ at $u \in U_i$ is only dependent on $i$, and is equal to
$\sigma_i (c)$. Thus $K(c) = \sum_{i=1}^{d_c} \sigma_i (c) \vol_{n-1}(U_i)$.

\begin{proposition}\label{propCCF4}
If $|K|$ is continuous at $c$, then $K$ is continuous at $c$.
\end{proposition}
\begin{myproof}
If $|K|(c) = 0$, then we immediately get $K(c)=0$, and so $K$ is
continuous at $0$.

Assume that $\Cal{U}_c$ is not empty. Let $U_c$ be a connected component of $\Cal{U}_c$. Then for each
$u\in U$, the number of connected components of $\Psi_f^{-1}(\{u\}\times \R)$ meeting $F_c$ is constant,
say equal to $l \geqslant 1$. Let $\Gamma_1 (u),\ldots,\Gamma_l(u)$ be these connected components. Given $u
\in \Cal{U}_c$, we know that there exists $\varepsilon >0$ such that for any $t \in
]c-\varepsilon,c[\cup]c,c+\varepsilon[$, we get
\begin{center}
$\Psi_f^{-1}(\{u\}\times \R) \cap F_t = \cup_{i=1}^l(\Gamma_i (u)\cap F_t)$.
\end{center}

From Proposition \ref{propGMTF2}, we know that the tangent Gauss degree of $f$, that is the degree of
$\nu_t $ at $u$ is a locally constant function of $(u,t)$. Thus for $u \in \Cal{U}_c$ given, there is
$\varepsilon>0$ as above such that for any $t \in ]c-\varepsilon,\varepsilon[$, $\deg_u \nu_t = \deg_u
\nu_c$. Let $\widetilde{\Cal{V}}$ be the connected component of $\widetilde{\Cal{U}}$ that contains $(u,c)$. Thus
$\widetilde{\Cal{V}} \cap \Sr^{n-1}\times\{c\} = U_c \times\{c\}$ is connected. So there exists $\varepsilon
>0$ such that for any $t \in ]c-\varepsilon,c+\varepsilon[$, $\widetilde{\Cal{V}} \cap \Sr^{n-1}\times\{t\} =
U_t \times\{t\}$ is connected, so $U_t$ is connected. Since $|K|$ is continuous, this implies that $\lim_{t
\fl c} \vol_{n-1}(U_t) = \vol_{n-1}(U_c)$. Since the Gauss tangent degree is constant on $\widetilde{\Cal{V}}$,
this implies the continuity at $c$ of the function total curvature.
\end{myproof}

The converse of this result is not true as shown in the following
example: take $f(x,y) = y(2x^2y^2 -9xy +12)$ as a function on the
real plane. Then we find that
\begin{center}
$\lim_{t\rightarrow c^-} |K|(c) = \lim_{t\rightarrow c^+} |K|(c) =
2\pi$, while  $|K|(0)=0$,
\end{center}
but $K$ is continuous at $0$.

\begin{remark}\label{rmkCCF2}
Talking about the discontinuity of $|K|$ or $K$ makes only sense at
values $c$ that are taken by the function $f$.
\end{remark}

Let us return to the continuity of the curvature of $f$. Let $c$ be a regular value. For each 
$l =1,\ldots,N_f$, the Hausdorff limit of $\lim_{t\rightarrow c^*}\clos (\Cal{U}_{l,c})$ is denoted 
by $\Scr{V}_l^*$, where $* =+$ or $*=-$.

\begin{proposition}\label{propCCF5}
Let $c$ be a regular value at which $|K|$ is not continuous. There exists an open subset $U \subset
\Sr^{n-1}$, such that for any $u \in U$, there exists a connected component $\Gamma$ of $\Gamma_u^+ (f):=
\Psi_f^{-1}(\{u\}\times \R)$, such that $\Gamma \cap f^{-1}(c)$ is empty and one of the two cases below
happens

(i) If $c$ is the infimum of $f$ along $\Gamma$, then for any
$\varepsilon >0$ small enough, $\Gamma \cap
f^{-1}(]c,c+\varepsilon[)$ is not bounded.

(ii) If $c$ is the supremum of $f$ along $\Gamma$, then for any
$\varepsilon >0$ small enough, $\Gamma \cap
f^{-1}(]c-\varepsilon,c[)$ is not bounded.
\end{proposition}
\begin{myproof}
Assume that $\lim_{t\rightarrow c^-}|K|(t) >|K|(c)$. The other
situation is absolutely similar.

Let $\Cal{U}_c = \sqcup_{i=1}^{d_c} U_i$, where  $U_i$ is a connected component of $\Cal{U}_c$. For
each $i$, we have $U_i \subset \Cal{U}_{s(i),c}$. Since $|K|(t) = \sum_{l=1}^{N_f} l \vol_{n-1}
(\Cal{U}_{l,t})$, we thus find that $\lim_{t\rightarrow c^-}|K|(t) = \sum_{l=1}^{N_f} l \vol_{n-1}
(\Scr{V}_l^-)$.

For each $i=1,\ldots,d_c$, there exists a positive integer $l(i) \geqslant s(i)$ such that $U_i \subset
\Scr{V}_{l(i)}^-)$. Thus the following is happening: \\
either \\
(a) there exists $i \in \{1,\ldots,d_c\}$ such that $l(i) > s(i)$ \\
or \\
(b) there exists $i \in \{1,\ldots,d_c\}$ such that $\vol_{n-1}
(\Scr{V}_{l(i)}^-) > \vol_{n-1} (U_i)$.

\smallskip
Assume that a phenomenon of type (a) contributes to the discontinuity of $|K|$ at $c$. Then for each $u \in
U_i$, there exists $\varepsilon >0$ such that $\Gamma_u^+ (f) \cap
f^{-1}(]c-\varepsilon,c+\varepsilon[)$ has $m(i)\geqslant l(i)$, connected components
$\Cal{G}_1,\ldots,\Cal{G}_{m(i)}$ such that, for $j = 1,\ldots,s(i)$, $f(\Cal{G}_j) =
]c-\varepsilon,c+\varepsilon[$, for $j = s(i)+1,\ldots,l(i)$, $f(\Cal{G}_j) = ]c-\varepsilon,c[$ and for $j
= l(i)+1,\ldots,m(i)$, $f(\Cal{G}_j) = ]c,c+\varepsilon[$.

Let $j \in \{s(i)+1,\ldots,l(i)\}$. If $\Cal{G}_j$ was bounded, its closure in $\Rn$ would then be 
$\clos (\Cal{G}_j) = \Cal{G}_j
\cup\{x_{c-\varepsilon}, x_c\}$, where $x_{c-\varepsilon} \in F_{c-\varepsilon}$ and $x_c \in F_c$. Since
$c$ is a regular value, this would mean that $\nf (x_c) = u |\nf (x_c)|$. Such a $u$ is a regular value of
$\nuf|_{f^{-1}(]c-\varepsilon,c+\varepsilon[)}$, thus $\Cal{G}_j$ could be extended to $\{f>c\}$ to a regular
curve through $x_c$, which would contradict $f(\Cal{G}_j) = ]c-\varepsilon,c[$. So we get that
$\Cal{G}_j$ never meet $f^{-1}(c)$. The same works for $j \in \{l(i)+1,\ldots,m(i)\}$.

\smallskip
Case (b) is proved similarly.

\smallskip
There is a last thing to say about the case $\Cal{U}_c$ is empty,
meaning that $\nu_c$ is of rank at most $n-2$. In this situation the
discontinuity of $|K|$ is created by phenomenon of type (b) only.
\end{myproof}

When the discontinuity of $|K|$ arises at a critical value $c$, it
is almost impossible to say anything similar to the previous
statement, in the general frame we are given.
%
%
%
%
%
%
%        **********************************************************************
%
%
%
%
%
%
%
\section{Total $\lambda$-curvature and total absolute $\lambda$-curvature}\label{sectionTLCTALC}
Given a compact connected manifold $M$ of dimension $m$, given a Morse function $g:M \mapsto \R$, let
$C_\lambda (g)$ be the set of critical points of index $\lambda$. The weak Morse inequalities state

\smallskip
\begin{center}
$\# C_\lambda (g) \geqslant b_\lambda (M)$, the $\lambda$-th Betti number.
\end{center}

\medskip
Let us consider now that $M$ is an orientable connected definable $C^l$, $l\geqslant 2$, hypersurface of
$\Rn$, with orientation map $\nu_M$.

Using Definition \ref{defCACH1}, given $\lambda \in \{0,\ldots,n-1\}$, let us define

\smallskip
\begin{center}
$\Cal{I}_M (\lambda) := \{x \in M : \lambda_M (x) = \lambda\}$.
\end{center}

\medskip
Since $\Cal{I}_M (\lambda) \subset M \setminus \crit (\nu_M)$, and $M \setminus \crit (\nu_M) \ni x \mapsto
\lambda_M (x)$ is definable
and locally constant, the subset $\Cal{I}_M (\lambda)$ is definable.

\medskip
Now we can define
\begin{definition}
Given $\lambda \in \{0,\ldots,n-1\}$, \\
(1) the total $\lambda$-curvature of $M$ is $K_M(\lambda) := \displaystyle{\int_{\Cal{I}_M (\lambda)} k_M
(x) \ud v (x)}$,

\smallskip
\noindent
(2) the total absolute $\lambda$-curvature of $M$ is $|K|_M(\lambda) :=
\displaystyle{\int_{\Cal{I}_M (\lambda)} |k_M| (x) \ud v (x)}$.
\end{definition}

If $M$ is compact then averaging on the restriction to $M$ of all the oriented linear projections 
we deduce

\begin{center}
 $|K|_M(\lambda)\geqslant {\rm vol}_{n-1} (\nu_M (\Cal{I}_M(\lambda))) \cdot b_\lambda (M)$.
\end{center}

\medskip
Now let us come to the case of definable functions. Let $f : \Rn \mapsto \R$ still be a $C^l$, 
definable function with $l\geqslant 2$. We use  the notations of Section \ref{sectionGMTF} and Section \ref{sectionHLGI}. Given any $t \notin
K_0 (f)$, let $\nu_t$ be the restriction of $\nu_f$ to $F_t$, and let $k_t (x)$
be the Gauss curvature at $x \in F_t$.

\medskip
Let us define the open definable subset

\medskip
\begin{center}
$\Cal{I}_f (\lambda):= \{x \notin  \crit (f) \cup \crit (\nuf): \mbox{ the index of } \ud_x
\nuf|_{T_xF_{f^{-1}(x)}} \mbox{ is } \lambda \}$
\end{center}

\medskip
Let $\Cal{E}_t$ still be the set of connected components of $F_t$.
Given $t \notin K_0 (f)$ and $\lambda \in \{0,\ldots,n-1\}$, let
\begin{center}
$\Cal{I}_t (\lambda) = \cup_{E \in \Cal{E}_t} \Cal{I}_E (\lambda) = \Cal{I}_f (\lambda) \cap F_t$.
\end{center}

\medskip
The family $(\nu_t(\Cal{I}_t (\lambda))_{t \notin K_0 (f)})$ is a definable family of subsets of $\Sr^{n-1}$.

\medskip
We define two new functions of $t$,

the {\it total $\lambda$-curvature} of $f$ :
\begin{center}
 $K(\lambda;t) := \sum_{E \in \Cal{E}_t} K_E (\lambda) = \displaystyle{\int_{\Cal{I}_t
(\lambda)} k_t (x) \ud v (x)}$,
\end{center}

and the {\it total absolute $\lambda$-curvature} of $f$ :
\begin{center}
$|K|(\lambda;t) := \sum_{E \in \Cal{E}_t} |K|_E (\lambda) = \displaystyle{\int_{\Cal{I}_t (\lambda)}|k_t| (x)
\ud v (x)}$.
\end{center}

Now we can state the main result of this section
\begin{theorem}
Let $\lambda \in \{1,\ldots,n-1\}$ be given. \\
(1) Let $c$ be a value. Then $\lim_{t \fl c^-}K(\lambda;t)$, $\lim_{t \fl c^+}K(\lambda;t)$, $\lim_{t \fl
c^-}|K|(\lambda;t)$ and $\lim_{t \fl c^+}|K|(\lambda;t)$ exist. \\
(2) Let $c$ be a value. Then
\begin{center}
$|K|(\lambda;t) \leqslant \min \{\lim_{t \fl c^-}|K|(\lambda;t),\lim_{t \fl
c^+}|K|(\lambda;t)\}$.
\end{center}
(3) The function $t \to |K|(\lambda;t)$ admits at most finitely many discontinuities. \\
(4) If the function $t \to |K|(\lambda;t)$ is continuous at $c$, then so is the function $t \to
K(\lambda;t)$.
\end{theorem}
\begin{proof}
The proof of each point works as the proof of the similar statement given for the functions $t\to K(t)$
and $t\to |K|(t)$. The reason for that is to consider the function $\Psi_\lambda := \Psi_f|_{\Cal{I}_f
(\lambda)}$, and then do exactly the same work as that done in Section \ref{sectionHLGI} and
\ref{sectionCCF}
\end{proof}
%
%
%
%
%
%
%        **********************************************************************
%
%
%
%
%
%
%
\section{Remarks, comments and suggestions}\label{sectionRCS}

In the introduction we mentioned that the original motivation was to try to find some equisingularity
conditions on the family of levels $(F_t)$ that could be read through the continuity
of these total curvature functions. We were especially interested in the problem
caused by the regular values that are also bifurcation values.
This phenomenon is without any doubt caused by some curvature, in the broader sense of
Lipschitz-Killing curvature, accumulation at infinity (i.e. on the boundary of the domain).
Without any bound on the complexity of the singularity phenomenon occurring at infinity,
it would be very naive and wrong to hope that the continuity of these curvature functions we dealt with
is a sufficiently fine measure of the equisingularity of the levels $F_t$ nearby a regular value.
In some simple cases, see \cite{Gr}, they can provide sufficient conditions to ensure the equisingularity 
in a neighborhood of a given regular asymptotic critical value. But in whole generality we may also have 
to consider these higher order curvatures.

\medskip
Consider again the situation of Section \ref{sectionGMTF}: That is of $f : \Rn \mapsto \R$
a definable function enough differentiable. We have associated to each regular value $t$ of
$f$, two real numbers, namely $K(t)$ and $|K| (t)$.

\medskip
Let $t$ be a regular value of $f$ and let $x \in F_t$. Let $q \in \{1,\ldots,n-1\}$ be a given integer.
Let $N$ be a $q$-dimensional sub-vector space of $T_x F_t$.
Let us denotes $k_t (N,x)$ the Gauss curvature at $x$ of $F_t \cap (x + (\R \nu_t (x) \oplus N))$
a hypersurface of the $q+1$-dimensional affine subspace $x + (\R \nu_t (x) \oplus N)$.
The $q$-th Lipschitz-Killing curvature of $F_t$ at $x$ is
the real number 
\begin{center}
\medskip
$\displaystyle{LK_q (x) : = \int_{G(q,T_xF_t)} k_t (N,x) \ud N}$, 
\medskip
\end{center}
where $\ud N$ is the volume form of $G(q,T_xF_t)$ the Grassmann manifold of $q$-dimensional sub-vector 
spaces of $T_x F_t$. 

\medskip
So we would like to define two real numbers: 

\smallskip
\noindent
the total $q$-th Lipschitz-Killing curvature,

\begin{center}
\medskip
$\displaystyle{L_q (t) = c_{n,q} \int_{F_t} LK_q (x) \ud v_{n-1} (x)}$,
\medskip
\end{center}
\noindent
and the total absolute $q$-th Lipschitz-Killing curvature, also called total $q$-th length
\begin{center}
\medskip
$\displaystyle{|L|_q (t) = c_{n,q} \int_{F_t} |LK_q (x)| \ud v_{n-1} (x)}$,
\medskip
\end{center}
\noindent
where $c_ {n,q}$ is a universal constant depending only on $q$ and $n$.

\medskip
In general, once $q <n-1$, with the level $F_t$ non compact, the number $|L|_q (t)$ is likely to be
infinite !

\medskip
A more relevant question would be to estimate the first (and may be the second) dominant term of 
the asymptotic of
\begin{eqnarray*}
|L|_q (t;R) & = & c_{n,q} \displaystyle{\int_{F_t \cap \B_R^n}} |LK_q (x)| \ud v_{n-1} (x) \\
L_q (t;R) & = & c_{n,q} \displaystyle{\int_{F_t \cap \B_R^n}}LK_q (x) \ud v_{n-1} (x)
\end{eqnarray*}
as $R \fl + \infty$ and to see how these dominant terms varies in $t$ or in $(t,R)$.
But as already said, even in the subanalytic category where results from \cite{CLR, LR} provide 
some general information on the nature of such terms, it is likely to be a very difficult question !
%
%
%
%
%
%
%
%     *****************************************************************
%
%
%
%
%
%
%
%
%\section*{Thanks}
%The author would like to thank the University of Bath (UK) and Carl von Ossietzky Universit\"at Oldenburg
%(Germany) for the working conditions provided while working on this note.
%We are pleased to thank D. Siersma, K. Kurdyka and J.M. Lion for  remarks, comments,
%suggestions and encouragements.
%
%
%
%
%     *****************************************************************
%
%
%     *****************************************************************
%
%
%
%
%


\begin{thebibliography}{referen} \label{biblio}
\bibitem[BB]{BB} A. Bernig\ and\ L. Br\"ocker, Courbures intrins\`eques
dans les cat\'egories analytico-g\'eom\'etriques, Ann. Inst. Fourier (Grenoble)
{\bf 53} (2003), no.~6, 1897--1924.

\bibitem[Br]{Br} L. Br\"ocker, Families of semialgebraic sets and limits,
in {\it Real algebraic geometry (Rennes, 1991)}, 145--162, Lecture Notes in Math., 1524, Springer, Berlin.

\bibitem[BK]{BK} L. Br\"ocker \& M. Kuppe, \em Integral Geometry of Tame Sets,
\rm Geom. Dedicata, {\bf 82} (2000) 285--323.

\bibitem[Com]{Com} G. Comte, \em \'Equisingularit\'e r\'eelle: nombres de Lelong et images  polaires,
\rm Ann. Sci. École Norm. Sup. (4) {\bf 33} (2000), no.~6,  757--788.

\bibitem[Cos]{Cos} M. Coste, \em An Introduction to O-minimal Geometry,
\rm Istituti Editoriali e Poligrafici Internazionali, Pisa (2000), 82 pages, also available on
\\
{\tiny http://perso.univ-rennes1.fr/michel.coste/polyens/OMIN.pdf}

\bibitem[CLR]{CLR} G. Comte \& J.-M.Lion \& J.-P. Rolin,
\em Nature log-analytique du volume des sous-analytiques, \rm Ill. J. Math. {\bf 44} (2000), No.4,
884--888.

\bibitem[vD1]{vdD1} L. van den Dries, Tame topology and o-minimal structures,
{\it Cambridge University Press, Cambridge}, 1998.

\bibitem[vD2]{vdD2} L. van dan Dries, Limit sets in o-minimal
structure, in \em Proceedings of the RAAG Summer School Lisbon 2003: O-minimal Structures, \rm (2005),
172--214, available on
{\tiny http://www.uni-regensburg.de/Fakultaeten/nat\_Fak\_I/RAAG/preprints/0159.html}

\bibitem[vDM]{vdDM}  L. van den Dries \& C. Miller, \em Geometric Categories and o-minimal strucutures,
\rm Duke Math. J., {\bf 84} (1996) 497--540.

\bibitem[Fu]{Fu}  J. Fu, \em Curvature Measures of Subanalytic Sets, \rm American J. Math., {\bf 116} (1994)
819--880.

\bibitem[Gr]{Gr} V. Grandjean, \em Tame Functions with strongly isolated singularities at infinity:
a tame version of a Parusi\'nski's Theorem, \rm to appear in Geometriae Dedicata, preprint available on 
{\tiny http://arxiv.org/PS{\_}cache/arxiv/pdf/0708/0708.2474v3.pdf}

\bibitem[Ka]{Ka} T. Kaiser, \em
On convergence of integrals in o-minimal structures on archimedean real closed fields, \rm Ann. Pol. Math.
{\bf 87}, (2005) 175-192.

\bibitem[Le]{Le} O. Le Gal, \em Mod\`ele compl\'etude des structures o-minimales polynomialement
born\'ees) \rm Th\`ese de doctorat, Universit\'e de Rennes I (2006).

\bibitem[LR]{LR} J.M. Lion \& J.P. Rolin,
\em Volumes, feuilles de Rolle de feuilletages analytiques et th\'eor\`eme de Wilkie, \rm  Ann. Fac. Sci.
Toulouse Math. (6) {\bf 7} (1998), no.~1, 93--112.

\bibitem[LS]{LS} J.M. Lion \& P. Speissegger, \em A geometric
proof of the definability of Hausdorff Limits, \rm Sel. Math., New Ser. {\bf 10} (2004), no.~3, 377--390.

\end{thebibliography}
\end{document}